\documentclass[12pt]{amsart}
\usepackage{amsfonts,amssymb,amsmath,verbatim,url,amsthm,calc,fullpage}
\usepackage{graphicx,psfrag}
\title{First-fit coloring on interval graphs has performance ratio at least 5}

\author{H.~A.~Kierstead$^\MakeLowercase{a}$}
\thanks{$^a$School of Mathematical and Statistical Sciences, Arizona State
University, Tempe, AZ 85287, USA. E-mail address: hal.kierstead@me.com.
Research was supported in part by NSF grant DMS-0901520.}

\author{David A.~Smith$^\MakeLowercase{b}$}
\thanks{$^b$School of Mathematical and Statistical Sciences, Arizona State
University, Tempe, AZ 85287, USA. E-mail address: 
David.A.Smith.1@asu.edu.
Research was supported in part by NSF grant DMS-0901520.}

\author{W.~T.~Trotter$^\MakeLowercase{c}$}
\thanks{$^c$School of Mathematics, Georgia Institute of Technology, Atlanta, GA 30332, USA. E-mail address: trotter@math.gatech.edu.}
\thanks{Copyright 2015, Elsevier. Licensed under the Creative Commons Attribution-NonCommercial-NoDerivatives 4.0 International, http://creativecommons.org/licenses/by-nc-nd/4.0/}

\newtheorem{thm}{Theorem}
\newtheorem{prop}[thm]{Proposition}
\newtheorem{lemma}[thm]{Lemma}

\theoremstyle{definition}

\newtheorem*{defn}{Definition}

\theoremstyle{remark}

\newtheorem*{dasnote}{Note}

\begin{document}
\begin{abstract}
First-fit is the online graph coloring algorithm that  
considers vertices one at a time in some order and assigns each  
vertex the least positive integer not used already on a neighbor.
The maximum number of colors used by first-fit on graph $G$ 
over all vertex orders is denoted $\chi_{FF}(G).$

The exact value of $R:=\sup_G\frac{\chi_{FF}(G)}{\omega(G)}$ over interval 
graphs $G$ is unknown.
Pemmaraju, Raman, and Varadarajan (2004) proved
$R \le 10$, and this can be  
improved to $8.$ 
Witsenhausen (1976) and
Chrobak and \'Slusarek (1988) showed $R\ge 4$, and \'Slusarek (1993) 
improved this to $4.45.$ We prove $R\ge 5.$

\end{abstract}

\maketitle


\section{Introduction}\label{intro}

A \emph{coloring} of graph $G=(V,E)$ is a function $f\colon V \to \mathbb{Z}$ where $f(u)\ne f(v)$ for all $uv\in E$.
We consider only finite graphs.
The \emph{clique size} of graph $G$ is the number $\omega(G)$ 
of vertices in a largest complete subgraph of $G.$ 
Every coloring of 
graph $G$ has at least $\omega(G)$ 
colors.

An \emph{interval} is a convex set of real numbers.
Graph $G=(V,E)$ is an \emph{interval graph} if there is a function that assigns to each $v\in V$
an interval $I_v$ so that 
\[E = \left\{uv\in \binom{V}{2} \middle\vert I_u\cap I_v\ne \emptyset\right\}.\]
Such a function $v\mapsto I_v$ is an \emph{interval representation.}
Every interval
graph $G$ has a coloring of just $\omega(G)$ colors. One can be constructed by
ordering vertices by interval left end and coloring by first fit.

In some applications, however, we begin assigning colors before the whole 
graph is seen.
A coloring algorithm is \emph{online} if, given graph $G$ on vertices
$v_1,\ldots,v_n$, it assigns (irrevocably)
for each $k\in\{1,\ldots,n\}$ a color to $v_k$ that
depends only on the subgraph of $G$ induced by seen vertices 
$\{v_j\mid 1\le j\le k\}.$
For example, the \emph{first-fit} algorithm produces coloring $f$ as follows: 
for $k$ from 1 to $n$, let
$f(v_k)$ be the least positive integer available, i.e., not a member of 
\[\{f(v_j) \mid 1\le j \le k\text{ and } v_jv_k \text{ is an edge}\}.\]
The coloring produced by this algorithm depends on vertex order.
For example, Figures~\ref{f01}
\psfrag{v1}{$v_1$}
\psfrag{v2}{$v_2$}
\psfrag{v3}{$v_3$}
\psfrag{v4}{$v_4$}
\psfrag{v5}{$v_5$}
\psfrag{v6}{$v_6$}
\psfrag{v7}{$v_7$}
\psfrag{v8}{$v_8$}
\psfrag{n0}{$0$}
\psfrag{n-1}{$-1$}
\psfrag{n-2}{$-2$}
\psfrag{n1}{$1$}
\psfrag{n2}{$2$}
\psfrag{n3}{$3$}
\psfrag{n4}{$4$}
\psfrag{n5}{$5$}
\psfrag{n6}{$6$}
\psfrag{n7}{$7$}
\psfrag{n8}{$8$}
\psfrag{n9}{$9$}
\psfrag{t1}{{\tiny $1$}}
\psfrag{t2}{{\tiny $2$}}
\psfrag{t3}{{\tiny $3$}}
\psfrag{t4}{{\tiny $4$}}
\begin{figure}
\begin{center}\includegraphics{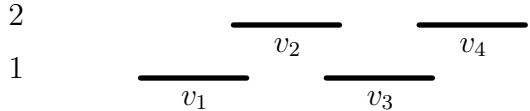}\end{center}
\caption{First-fit uses 2 colors on $P_4=v_1v_2v_3v_4$}\label{f01}
\end{figure}
and \ref{f02}
\begin{figure}
\begin{center}\includegraphics{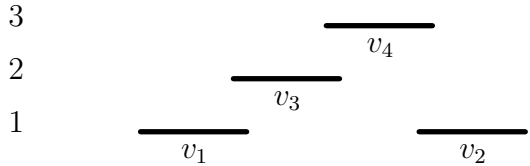}\end{center}
\caption{First-fit uses 3 colors on $P_4=v_1v_3v_4v_2$ }\label{f02}
\end{figure}
show first-fit colorings of the 4-vertex path.

First-fit may be the first algorithm that comes to mind when the goal is
to use few colors.
How wasteful is it in the worst case?
Let $\chi_{FF}(G)$ be the maximum 
over all vertex orders of the number of colors it uses on graph $G$.
For example, $\chi_{FF}(P_4)=3.$

Kierstead \cite{HK} showed that $\chi_{FF}(G) \le 40\omega(G)$ for every 
interval graph $G$. 
This was improved to $25.72\omega(G)$ by Kierstead and Qin \cite{KQ},
then $10\omega(G)$ by Pemmaraju, Raman, and Varadarajan \cite{PRV}.
Brightwell, Kierstead, and Trotter (unpublished)
and Narayanaswamy and Subhash Babu \cite{NSB} observed that the technique in \cite{PRV} yields $8\omega(G)$. 
These results provide bounds above 
the worst-case performance ratio of the first-fit algorithm on interval graphs,
\[R:= \sup\left\{\frac{\chi_{FF}(G)}{\omega(G)} \middle\vert G\text{ is a nonempty interval graph}\right\}.\]
So $1\le R \le 8.$ 
We want to know the exact value of $R$. 
This motivates a search for lower bounds.

Already $1.5 \le R$ due to the example $P_4.$ 
Variations on this example yield improvements, and
this is the structure of the present paper:
we recapitulate results of earlier authors in Section~\ref{s2}, 
generalize their constructions in Section~\ref{s3}, and tune the new construction to get $5\le R$
in Section~\ref{rseq}.
In Section~\ref{nocap} we show that this lower bound is the best that our approach allows.
A few more historical notes on related problems appear in Chapter 1 of the  
dissertation of Smith \cite{Smith}.
\section{Walls and caps}\label{s2}

Figure~\ref{f02}
is a certificate of a bound below $R.$
Pertinent features include total number of colors, clique size, and the set of colors neighboring each vertex, especially colors below its own.
Compare this to a brick wall with bricks arranged in courses. We may be concerned with its height, with its density, and, for stability, that each brick be supported by something below.

Gy\'arf\'as and Lehel \cite{GL} defined a wall to be essentially a graph and vertex order.
A linear order is overprecise for our purpose.
More suitable is a
partial order.
For example, the first-fit algorithm
uses 3 colors on the 4-vertex path
whenever the path's end vertices precede its inner 
vertices.

Colors can be determined during the building of a wall.
Suppose for all $k$ that the levels of the neighbors of each vertex in level $k$ include $1,\ldots,k-1$, but not $k$ (as in Figures~\ref{f01}, \ref{f02}, and \ref{f03}).
When we present such a wall to the first-fit algorithm from bottom to top,
the algorithm gives color 1 to the first level, color 2 to the second level, and so on.
First-fit colors are determined by this structural condition known as support.
We get large bounds below $R$ from tall walls with small cliques and supported vertices.

\begin{defn}
Let $N(v)$ denote the set of neighbors of vertex $v$, and $N[v] := N(v) \cup \{v\}.$
When $f$ is a function on domain $V$, and $U\subseteq V,$ we denote the image 
of $U$ by $f(U).$
Pair $(G,f)$ is an \emph{$r$-wall} if:
\begin{enumerate}
\item $G$ is an interval graph on vertex set $V$
with a given interval representation;
\item\label{wc2} $f\colon V \to \{1,2,\ldots\}$ is a coloring of $G$;
\item\label{wc3} $f(N[v]) \supseteq \{1,\ldots,f(v)\}$ for all $v\in V$ (\emph{support}); and
\item $|f(V)|\ge r\omega(G).$
\end{enumerate}
\end{defn}
For example, 
Figure~\ref{f01} depicts a $1$-wall.
Figure~\ref{f02} depicts a $1.5$-wall.
Figure~\ref{f03} depicts a $2$-wall.
We may refer to an $r$-wall as a \emph{wall}.
A wall is \emph{nonempty} when $V$ is nonempty.
The \emph{clique size} of a wall is the clique size of its graph.
A \emph{wall on $t$ colors} is a wall with $|f(V)|=t.$
\begin{prop}
If a nonempty $r$-wall exists, then $r \le R.$
\end{prop}
\begin{proof}
Given $r$-wall $(G,f)$ where $V$ is the vertex set of $G$, 
let $v_1,\ldots,v_n$ be an enumeration of $V$ that increases in color, 
i.e., $f(v_i)\le f(v_j)$ whenever $1\le i < j \le n.$ 
Then the first-fit algorithm (applied to this order) assigns color $f(v)$ 
to $v$ for each $v\in V.$ 
In particular, $\chi_{FF}(G) \ge f(v_n)\ge r\omega(G).$
\end{proof}
So $2\le R$ by Figure~\ref{f03}.
\begin{figure}
\begin{center}\includegraphics{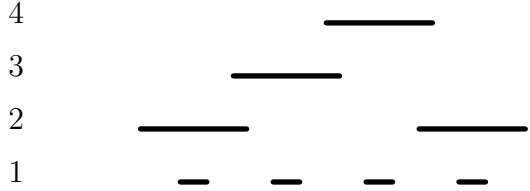}\end{center}
\caption{First-fit uses 4 colors on this interval graph of clique size 2 when vertices are ordered bottom to top, so $R \ge \frac{4}{2}.$}\label{f03}
\end{figure}

A given interval graph has many interval representations.
For example, if the transformation $x \mapsto x + 0.1$ is applied to each interval of a representation,
the graph structure is unchanged. The same is true of $x \mapsto -\sqrt{2} x$ and of every invertible affine transformation.
If we use intervals of finite length in our representations, we can move them by affine transformation
into any given interval of positive length 
while preserving graph structure.
This is what we mean by \emph{squeezing} an interval representation into an interval.

Consider the sequence $T_0,T_1,T_2,\ldots$ of walls depicted in 
Figure~\ref{f04}.
$T_0$ is a 1-vertex wall.
$T_1$ is the wall of Figure~\ref{f03}.
$T_2$ is the wall of Figure~\ref{f04_5}.
For $i > 0,$ wall $T_i$ is made as in Figure~\ref{f04}
by placing 4 new vertices in levels above 4 copies of $T_{i-1}$.
The copies, whose tops are represented in Figure~\ref{f04} by wedge-like cones, have their interval representations squeezed into narrow intervals away from the junctions of the intervals of the 
4 new top vertices, so that the clique size of the graph
of $T_i$ is just 1 greater than that of the graph of $T_{i-1}.$
Wall $T_i$ has $3i+1$ levels, and its graph has clique size $i+1$, so $3\le R.$
\psfrag{w}{$T_{i-1}$}
\psfrag{W}{$T_i:$}
\begin{figure}
\begin{center}\includegraphics{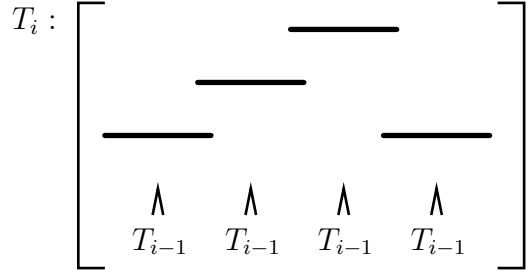}\end{center}
\caption{Wall $T_i$ is made by placing 4 intervals above 4 narrow copies of $T_{i-1}$. The construction yields $R\ge \frac{1}{1}, \frac{4}{2}, \frac{7}{3},\cdots \to 3.$}\label{f04}
\end{figure}
We generalize this arrangement of new vertices above previously constructed 
walls (as in Figure~\ref{f04})
by defining a wall-like object called a cap.
\begin{figure}
\begin{center}\includegraphics{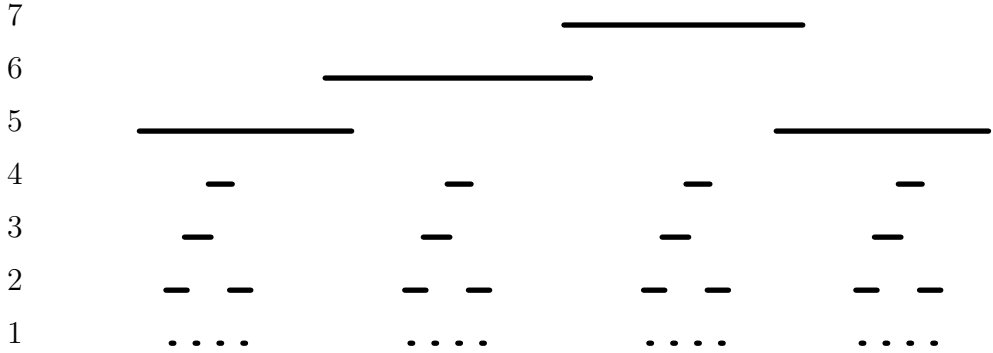}\end{center}
\caption{First-fit uses 7 colors on this interval graph of clique size 3 when vertices are ordered bottom to top, so $R \ge \frac{7}{3}.$ This is wall $T_2.$}\label{f04_5}
\end{figure}

\begin{defn}
Let $G$ be a (finite) nonempty interval graph on vertex set $V$ with interval 
representation $v \mapsto I_v$.
Let $f\colon V\to \{0,-1,-2,\ldots\}$ be a coloring of $G$.
(Caps are built from the top down. 
It may not be clear at first how many colors are needed.
So we use 0 for the top level and build the cap downward using 
negative integers.)
The remaining conditions for caps (yet to be defined precisely) are illustrated in 
Figure~\ref{f05}. For each vertex $v$:
\begin{enumerate}
\item Empty space is reserved for a previously constructed wall $W_v$.
This void is bounded horizontally by some interval $J_v\subseteq I_v.$ 
Vertically the void begins at some level $c_v$ below $f(v)$ and extends downward without bound.
\item Few cap vertices $u$ in levels above $c_v$ have intervals $I_u$ that meet $J_v.$
\item Each level between $f(v)$ and $c_v$ contains a (supporting) cap vertex that neighbors $v$.
\end{enumerate}
Here caps are defined precisely.
Let $r\ge 1.$
Suppose for each $v\in V$ there is an interval $J_v\subseteq I_v$ 
of positive length
(reserving horizontal space for a wall)
where $(\forall u,v\in V)[ J_u\cap J_v \ne \emptyset \to J_u=J_v].$
We denote the colors of vertices $u$ whose intervals $I_u$ meet $J_v$ by $C_v.$ That is,
\[C_v := f\left(\{u\in V \mid I_u\cap J_v \ne \emptyset\}\right).\]
\psfrag{Iv}{$I_v$}
\psfrag{Fv}{$f(v)$}
\psfrag{Iu}{$I_u$}
\psfrag{Fu}{$f(u)$}
\psfrag{Jv}{$J_v$}
\psfrag{Cv}{$c_v$}
\psfrag{zero}{$0$}
\psfrag{Wv}{space for $W_v$}
\psfrag{captop}{cap top}
\psfrag{conetop}{cone top}
\psfrag{SPA}{{\bf spa}}
\psfrag{RSE}{{\bf rse}}
\psfrag{EMPTY}{{\bf empty}}
\begin{figure}
\begin{center}\includegraphics{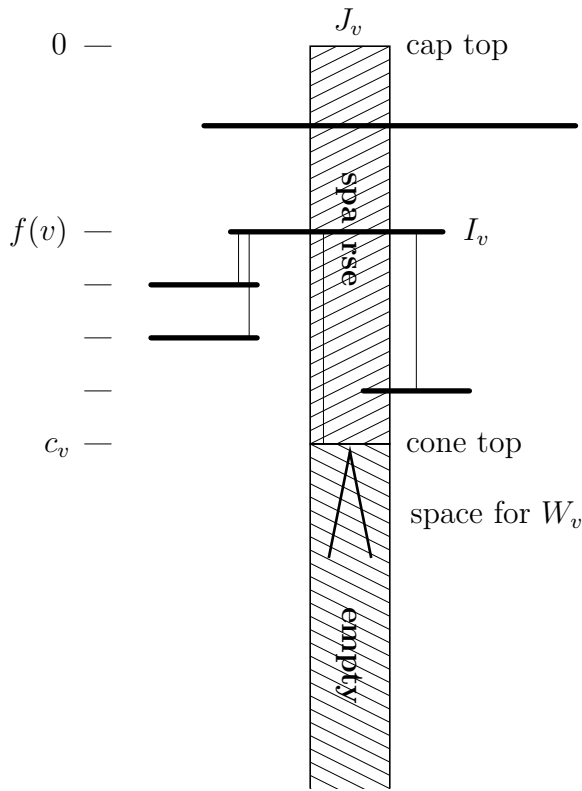}\end{center}
\caption{Detail of a cap at vertex $v$.
A cap is used to build a sequence of walls that tends to $r$ in the ratio of 
number of colors to clique size of the graph.
Space is reserved for a wall $W_v$ whose vertices will support $v$ when the cap is extended to a wall.
A wedge-shaped cone at the top of the void indicates where the top of $W_v$ will go
(but does not indicate how tall $W_v$ is).
Vertex $v$ is supported by cap vertices in levels from its own ($f(v)$) down to the top of the void ($c_v$).
Above the void is a sparse region.
At most $1/r$ of the levels from cap top to cone top have a vertex $u$ where $I_u$ meets $J_v.$
}\label{f05}
\end{figure}
If for each $v\in V$ there is a color $c_v\in \{-1, -2, \ldots\}$
where:
\[\begin{array}{lcll}
|C_v \cap (-\infty, c_v]| &=& 0 & \text{(\emph{emptiness});}\\
|C_v\cap (c_v, 0]| &\le& -c_v/r & \text{(\emph{sparseness});}\\
\mathbb{Z} \cap (c_v,f(v)] &\subseteq&f(N[v]) &\text{(\emph{support});}
\end{array}\]
and if $f(v)=0$ for some $v\in V$,
then $(G,f, v\mapsto I_v, v\mapsto J_v, v\mapsto c_v)$
is an \emph{$r$-cap}.
\end{defn}

\begin{prop}\label{tcap}
Let $r\ge 1.$
If an $r$-cap exists, then $r\le R$; specifically, there is a constant
$b$ so that for all $k\in\{0,1,\ldots\},$
there is a wall of clique size $k$ with at least $rk-b$ colors.
\end{prop}

Before the proof, here are some lemmas.

\begin{lemma}\label{cliquewall}
For $k \ge 0$ there is a wall on $k$ colors with clique size $k.$
\end{lemma}
\begin{proof}
Consider a wall on $k$ vertices with constant interval representation $v\mapsto [0,1]$.
\end{proof}

\begin{lemma}\label{walldrop}
If there is a wall on $t > 0$ colors whose graph has no $k$-clique,
then there is a wall on $t-1$ colors whose graph has no $k$-clique.
\end{lemma}
\begin{proof}
Delete all vertices of the highest color.
Alternatively, delete all vertices of the lowest color and renumber colors.
\end{proof}

\begin{proof}[Proof of Proposition~\ref{tcap}]
Let $\displaystyle c = \max_{v\in V}-c_v$ (the depth of the deepest cone top).
Let $b= rc.$
Argue by induction on $k$.
When
$0 \le k < c$,
Lemma~\ref{cliquewall} provides a wall.
This concludes the base case.

Now suppose $k\ge c$ for the inductive step.
Let $v$ be a cap vertex.
Note that $f(v)\in C_v$.
By the emptiness condition, $f(v) \in (c_v, 0].$
By the sparseness condition, $-c_v/r \ge |C_v \cap (c_v, 0]| \ge 1.$
And $-c_v/r \le -c_v \le c \le k.$
So $0 \le k - \lfloor -c_v/r\rfloor < k,$ and we may invoke the inductive hypothesis to obtain a wall $W_v$ of clique size 
$k-\lfloor -c_v/r\rfloor$ with at least
$r(k-\lfloor -c_v/r\rfloor)-b\ge rk+c_v-b$ colors.
We squeeze the interval representation of the graph of wall $W_v$ into interval $J_v.$ 
By the sparseness condition, this addition creates no clique on more than $k$ vertices.

In the cap, $v$ has neighbors of colors $f(v) - 1, f(v) - 2, \ldots, c_v + 1$.
We extend $f$ to $W_v$ by resuming the sequence of colors. We paused at $c_v+1$. We
continue by recoloring $W_v$ from its top down with colors $c_v, c_v - 1, c_v - 2, \ldots.$
The sequence lasts until levels of $W_v$ are exhausted.
The sequence extends to $c_v -(rk + c_v - b) + 1 = b - rk + 1$ or farther.
Note that $b - rk + 1$ does not depend on $v$.

After this is done for all cap vertices, each vertex has wall support down to color $b - rk + 1,$
and some particular cap vertex $v$ with $f(v) = 0$ has neighbors colored $-1, -2, \ldots, b - rk + 1.$
Vertices beyond (that is, of color at most $b - rk$) are dropped as in the proof of Lemma~\ref{walldrop}.
Now the colors used are $0, -1, -2, \ldots, b - rk + 1.$
Walls are supposed to have positive colors, so we translate the palette, adding $rk - b$ to each color used.
The result is a wall of clique size
at most $k$ with at least $rk-b$ colors.
We may increase the clique size to $k$ by placing a wall of Lemma~\ref{cliquewall} on the side.
\end{proof}

Witsenhausen \cite{Wits} proved $4\le R,$ as did
Chrobak and \'Slusarek \cite{CS}, who drew a picture much like
the 4-cap of Figure~\ref{f07}.
\psfrag{ta}{{\tiny (a)}}
\psfrag{tb}{{\tiny (b)}}
\psfrag{tc}{{\tiny (c)}}
\psfrag{td}{{\tiny (d)}}
\psfrag{te}{{\tiny (e)}}
\psfrag{tf}{{\tiny (f)}}
\psfrag{tg}{{\tiny (g)}}
\psfrag{th}{{\tiny (h)}}
\psfrag{ti}{{\tiny (i)}}
\psfrag{tj}{{\tiny (j)}}
\psfrag{tk}{{\tiny (k)}}
\psfrag{tl}{{\tiny (l)}}
\psfrag{tm}{{\tiny (m)}}
\psfrag{tn}{{\tiny (n).}}
\psfrag{t8}{{\tiny $8$}}
\psfrag{t12}{{\tiny $12$}}
\psfrag{t16}{{\tiny $16$}}
\begin{figure}
\begin{center}\includegraphics{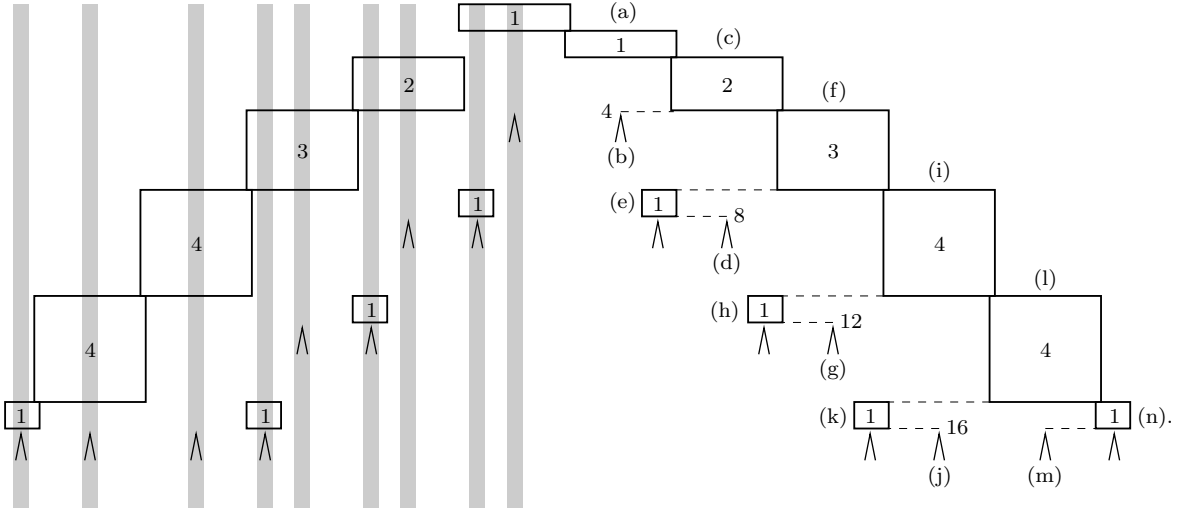}\end{center}
\caption{A $4$-cap. Each wedge-shaped cone marks the top of a void reserved for a previously constructed wall.
A wall is placed in each void when the cap is extended to a wall by Proposition~\ref{tcap}.
On the left are drawn tall gray bars, each corresponding to a hatched region (sparse and empty combined) of Figure~\ref{f05}.
A number $h$ inside a box 
indicates $h$ vertices with a common 
interval and $h$ consecutive colors.
}\label{f07}
\end{figure}
With Figure~\ref{f05} to remind us of cap conditions,
let us verify that Figure~\ref{f07} depicts a 4-cap.
Each box
has a cone beneath it.
Each cone has a void of positive width below.
When a vertical line segment is drawn from the top of the cap to the top of a cone, at most 25\% of the ink covers boxes.
Each box is supported by other boxes down to its cone.
So $4\le R$ by Proposition~\ref{tcap}.

We have verified that Figure~\ref{f07} depicts a 4-cap.
Now let us see how this 4-cap was invented, toward improving on it.
How do we arrange boxes and cones to meet the cap conditions?
The plan roughly is this:
\begin{itemize}
\item A cap cannot be empty, so start with a box.
\item Each box needs a cone beneath it, placed low due to the sparseness condition.
\item When a cone is too low to support its box alone, add supporting boxes. A cone divides the lower space into left and right, so add a left supporter, a right supporter, or both.
\item Keep the left and right halves of the cap similar.
\end{itemize}
Here we execute the plan.
\begin{enumerate}
\item[(a)] Twin boxes 
of height 1 (\emph{$1$-tall}) go at the top.
One is marked (a) in Figure~\ref{f07}.
The rest of the cap is the same left and right. 
We describe the right.
\item[(b)] A cone marked (b) in Figure~\ref{f07} goes below the 1-tall box at depth $4\cdot 1 = 4$, suited to the target ratio
$r=4.$ 
That is, we place the cone high in order to keep small the number of levels in which the 1-tall box above it requires support; 
but not so high as to violate the sparseness condition. 
In particular, if the cone is positioned $d$ units from the cap top, then the sparseness condition
implies $d\ge 4,$ so we choose $d=4$.
The 1-tall box still requires support in 2 levels.
This is addressed immediately.
\item[(c)] A 2-tall box is placed to fill the vertical gap between the bottom of box (a) and its cone (b). We are filling the gap to
satisfy the cap support condition for intervals of box (a).
\item[(d)] A cone goes below box (c) at proper depth ($4\cdot 2=8$). A gap of 4 colors remains between the new cone and the 2-tall box it supports.
\item[(e)] A 1-tall box (and cone) fits on the inside, leaving a gap of only 3 on the outside.
The idea is to distribute the support of box (c) between two others, boxes (e) and (f), left and right of the cone under box (c). 
Just left of box (c), there is only 1 unit of ``weight'' above (i.e., box (a)), and yet we are at depth 8, yielding the opportunity to add another unit of weight and still satisfy the sparseness condition.
\item[(f)] A 3-tall box fills the gap. 
\item[(g)] A cone goes below. A 5-gap remains (between cone and 3-tall box).
\item[(h)] A 1-tall box fits inside.
\item[(i)] A 4-tall box fills the gap.
\item[(j)] A cone goes below. A 5-gap remains.
\item[(k)] A 1-tall box fits inside.
\item[(l)] A 4-tall box fills the gap.
\item[(m)] A cone goes below. A 1-gap remains.
\item[(n)] A 1-tall box on the outside finishes the construction.
\end{enumerate}

Can we improve this result by the same procedure with $r=5$?
Let us record the sequence $u_1, u_2, \ldots$ of box heights along the outer
strand, which
for $r=4$ is
$1,2,3,4,4$. 
The procedure halts because
the last number does not exceed its predecessor.
For $r=5$ the sequence is $1,3,8,22,61,170,475,1329,3721,10422,29196,\ldots .$
We cannot complete the construction if the sequence increases forever. Does it?

Crucially, sequence
 $u_1,u_2,\ldots$ 
obeys a linear recurrence.
This is depicted in Figure~\ref{f08}.
\psfrag{un-3}{$u_{n-3}$}
\psfrag{un-2}{$u_{n-2}$}
\psfrag{un-1}{$u_{n-1}$}
\psfrag{un}{$u_n$}
\psfrag{un-2-un-3}{$u_{n-2}-u_{n-3}$}
\psfrag{un-1-un-2}{$u_{n-1}-u_{n-2}$}
\psfrag{run-2}{$ru_{n-2}$}
\psfrag{run-1}{$ru_{n-1}$}
\psfrag{B}{$B$}
\psfrag{OUTERCOURSE}{outer strand}
\begin{figure}
\begin{center}\includegraphics{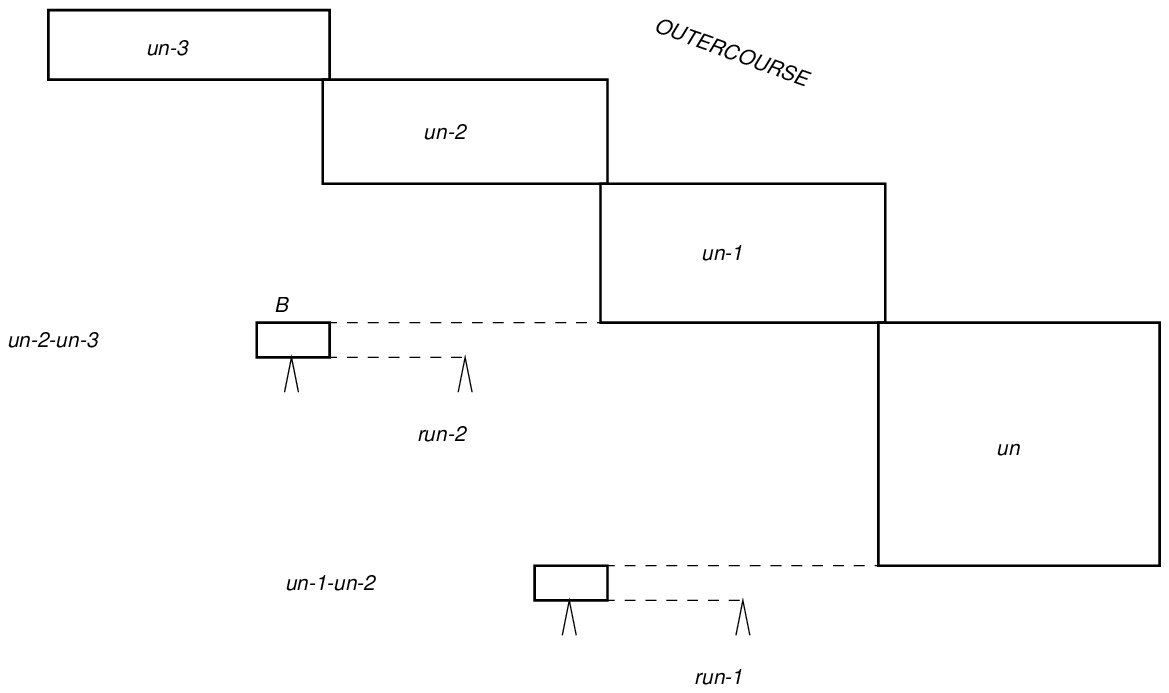}\end{center}
\caption{Obtain a recurrence by comparing depths of cones under two consecutive boxes of the outer strand.}\label{f08}
\end{figure}
The recurrence follows 
from the procedure 
by a comparison of depths of cones under two consecutive boxes ($u_{n-2}$-tall and $u_{n-1}$-tall) of the outer strand.
When a box $B$ is placed on the inside to support the $u_{n-2}$-tall box, 
$B$ has no supporting box of its own.
So the cone under $B$ is at the same depth $ru_{n-2}$ as the cone under the $u_{n-2}$-tall box.
To satisfy the sparseness condition for $B,$ the total height of boxes above the cone under $B$ must be at most $u_{n-2},$ and in the procedure we choose equality.
Above $B$ is a $u_{n-3}$-tall box.
So $B$ is $(u_{n-2} - u_{n-3})$-tall.
Similarly a $(u_{n-1}-u_{n-2})$-tall box supports the $u_{n-1}$-tall box from the inside.
Comparing cone depths,
\[ru_{n-2}-(u_{n-2}-u_{n-3})+u_n + (u_{n-1}-u_{n-2}) = ru_{n-1},\] so
heights of boxes of the outer strand are given by the sequence
\begin{eqnarray*}
u_0 = u_1 &=& 1\\
u_2 &=& r-2 \\
u_n &=& (r-1)u_{n-1} - (r-2)u_{n-2}-u_{n-3}
\end{eqnarray*}
for $n\ge 3.$
What determines the eventual behavior of such sequences? The discriminant
\[D = -31+6r+7r^2-6r^3+r^4\] 
of the characteristic polynomial
$1-(r-1)x + (r-2) x^2+x^3$
of the recurrence is negative when $r=4$ and positive when $r=5.$
As we will see when we refine the construction, 
such distinctions are important to sequence behavior.

Meanwhile we may improve the bound $4\le R$ by considering 
noninteger $r$ as \'Slusarek \cite{Sl} did when he obtained $4.45\le R.$
Observe that if each number in Figure~\ref{f07} is doubled, say,
the result is again a cap.
Therefore we may use (temporarily) rational numbers instead of integers for 
box heights and depths.
$D$ has a root \[r_+ = 1.5 + 0.5 \sqrt{13 + 16\sqrt{2}} \approx 4.48.\]
The best
we can do with the rational generalization of the method of 
Chrobak and \'Slusarek is 
$4.48 \approx r_+ \le R$.
(Of course $r_+$ is irrational, but a sequence of rationals increasing to $r_+$
establishes $r_+ \le R.$)

\section{A new construction}\label{s3}

In the previous section, our attempt to obtain a lower bound $5\le R$ failed due to an unhalting procedure.
Yet a certain resource was not fully exploited: 
the cap grew only 1 box inward.
That is, boxes of the outer strand had 1 or 2 supporting boxes, but
inner boxes (e), (h), and (k) of Figure~\ref{f07} had no supporters.
We will enable more growth inward by a construction that generalizes the one of Chrobak and \'Slusarek.

Let us reformulate the cap conditions as they apply to our construction.
\begin{enumerate}
\item Intervals are bundled into boxes, where a box of height $h$ represents identical intervals in $h$ consecutive levels.
Distinct boxes meet pairwise in sets of area 0.
Intervals bundled in a common box share a cone.
\begin{enumerate}
\item \emph{Emptiness}:
Each box has a cone beneath it, representing a void for an older wall.
\item \emph{Sparseness}:
Each vertical line segment from cap top to cone top (note: cones have positive width, so for a given cone, there are infinitely many such line segments) meets boxes in portion at most $1/r$ of its length. 
If $d$ is the distance of cone top from cap top, and $w$ (for \emph{weight above} the cone) is the total length of the part that meets boxes,
then the sparseness condition is $d \ge rw.$
\item \emph{Support}:
Each box is supported by at most two others below it that cover the vertical extent from its bottom to its cone top.
\end{enumerate}

\item Depths of box features (i.e., tops and bottoms) need not be integral, but are rational. They can be made integral by scaling later.
\end{enumerate}

This defines a \emph{rational cap}.
Our construction is less direct than that of Chrobak and \'Slusarek in two ways.
First, there is no rational cap with $r=5$ (see Section~\ref{nocap}).
Instead we construct a sequence of rational caps with $r\to 5.$ 
Second, each rational cap is made in
a finite sequence of quasi-cap constructions.

\begin{defn}
The construction of the $4$-cap of Chrobak and \'Slusarek (Figure~\ref{f07}) begins 
with twin 1-tall boxes at the top, one supporting the other.
Apart from the difference in vertical position of the twin 1-tall boxes,
the left and right sides are mirror images.
Consequently, the lower 1-tall box has at most one supporter.
This supporter (if any) is the \emph{key box}, or the box in \emph{key position}.
In Figure~\ref{f07} the key box is a 2-tall box labeled (c).
Such a rational cap (with twin 1-tall boxes, symmetry, and at most 1 supporter of the lower twin) 
is \emph{normal.}
\end{defn}

\begin{defn}
In a normal rational cap, the purpose of a key box is to support a 1-tall box,
so the top of the key box meets the bottom of the 1-tall box.
A \emph{quasicap} is a normal rational cap, except that a quasicap always has a key box, and the top of the key box may fall short, leaving a vertical gap (as in Figures \ref{f09} and \ref{f10}).
Further, we require of a quasicap that apart from the twin 1-tall boxes, no box has its top higher than the top of the key box.
\end{defn}

\psfrag{tr}{$r$}
\psfrag{theta}{$\theta$}
\psfrag{1side}{${\scriptstyle \text{1 side}}$}
\psfrag{keyside}{${\scriptstyle \text{key side}}$}
\begin{figure}
\begin{center}\includegraphics{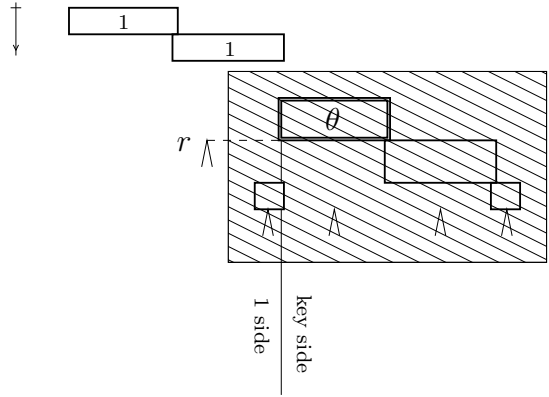}\end{center}
\caption{This quasicap has a $\theta$-tall box in the key position (bold).
When $\theta < r - 2,$ there is a gap between this box and the 1-tall box it should support.
The hatched portion appears in the next figure.
}\label{f09}
\end{figure}
\psfrag{u0}{{\tiny $u_0=1$}}
\psfrag{u1}{{\tiny $u_1=1$}}
\psfrag{u2}{$u_2=\theta + \delta$}
\psfrag{u3}{$u_3$}
\psfrag{uN+1}{$u_{N+1}$}
\psfrag{tru2}{$ru_2$}
\psfrag{tru3}{$ru_3$}
\psfrag{truN}{$ru_{N}$}
\psfrag{KEYBOX}{key box}
\psfrag{INNERCOPY1}{copy $1$}
\psfrag{INNERCOPY2}{copy $2$}
\psfrag{OUTERCOPY}{copy $N$}
\psfrag{THEGAP}{{\tiny This gap has shrunk by $\delta.$}}
\begin{figure}
\begin{center}\includegraphics{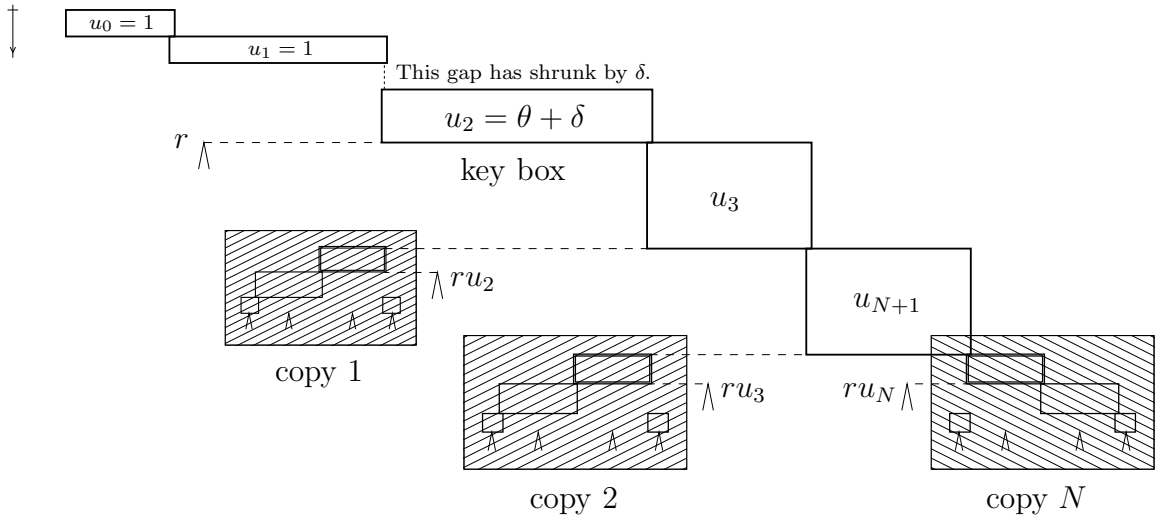}\end{center}
\caption{This quasicap has a $(\theta+\delta)$-tall box in the key position.
Pasted into place as shown are distorted copies of the hatched portion of the rational cap
from the previous figure.
(The main text defines the distortion and explains why the result is a quasicap.)
The goal is a quasicap with an $(r-2)$-tall box in the key position (that is, a rational cap).
The sequence of boxes in the outer strand depends on our choice $\delta$ (the size of our attempt to improve).
If $\delta$ is too large, then the sequence fails to terminate. 
}\label{f10}
\end{figure}

Our goal is to make a rational cap from a sequence of quasicaps where the vertical gap shrinks to nothing.
Figures \ref{f09} and \ref{f10} illustrate a gap-reducing step.
From a quasicap with a $\theta$-tall key box we make one with a $(\theta+\delta)$-tall key box.
Outer-strand box heights of the new quasicap depend on $\theta$ and $\delta$. A linear recurrence defines those heights:
\begin{eqnarray*}
u_0 = u_1 &=& 1\\
u_2 &=& \theta + \delta \\
u_n &=& (r-\theta)u_{n-1} - (r-2\theta)u_{n-2}-\theta u_{n-3}
\end{eqnarray*}
for $n\ge 3.$
In the special case $(\theta,\delta)=(1,r-3),$ this definition agrees with the sequence of outer-strand box heights
in the construction of Chrobak and \'Slusarek.
That construction can be regarded as an instance of ours
in the case of a single gap-reducing step from a trivial initial quasicap (with a 1-tall box
in the key position).

\begin{prop}\label{p3}
Fix some rational $r>4.$
Suppose there exists a quasicap with key box height $\theta \ge 1$ 
(henceforth a \emph{$\theta$-quasicap}). Suppose also $\theta < r-2.$
If there exist rational $\delta>0$ and integer $N>1$ so that 
the sequence $u_0,u_1,u_2,\ldots$ defined above satisfies
\[u_0 = u_1 \le u_2 < \cdots < u_N\ge u_{N+1},\]
then a $(\theta+\delta)$-quasicap exists.
\end{prop}
This result, proved in the many lemmas of this section, 
is a way to produce rational caps,
provided that improvements $\delta$ can be made until the gap shrinks to 0 (or equivalently, the key box grows to height $r-2$).

\begin{defn}
Here is the construction of Proposition~\ref{p3}.
The gap-reducing step is from an old quasicap (Figure~\ref{f09}) to a
\emph{new object} (Figure~\ref{f10}) which we will show to be a quasicap. 
\begin{enumerate}
\item As in Figure~\ref{f10}, lay out an outer strand of boxes of heights
$u_0,u_1,u_2,u_3,\ldots, u_{N+1}$ down the right side.
The $u_2$-tall box has its bottom at depth $r$, and for $n= 2,\ldots,N+1$ the $u_n$-tall box has its bottom at 
depth $r + u_3 + u_4 + \cdots + u_n.$
Place the cone of the $u_n$-tall box at depth $ru_n$ for $n=0,\ldots,N.$
Place the cone of the $u_{N+1}$-tall box at depth $ru_N.$
\item Paste into place under the outer strand (see Figure~\ref{f10})
copies $1,\ldots,N$ of part of the old quasicap (from Figure~\ref{f09}).
The part copied is only the right side of the old quasicap, excluding the twin 1-tall boxes.
These are not exact copies, but are stretched or squeezed by an amount calculated to satisfy the sparseness condition.
That the sparseness condition is satisfied is verified at the end of this section.
The stretching, squeezing, and moving of a given copy $n$ can be expressed in a cartesian coordinate system by an invertible affine plane transformation \[(x,y) \mapsto (\mu_n(x), \varphi_n(y)),\]
where 
\begin{enumerate}
\item the horizontal component $\mu_n$ of the transformation, itself an invertible affine transformation, is strictly decreasing for $n=1,\ldots,N-1$ and strictly increasing for $n=N$,
and 
\item the vertical component $\varphi_n$ is the invertible affine transformation
\[\varphi_n(y) = (u_{n+1}-u_n)y + ru_n\]
for $n=1,\ldots,N-1,$ with $\varphi_N = \varphi_{N-1},$
where $y$ is depth measured from cap top.
\end{enumerate}
Function $\mu_n$ is continuous; it keeps each box in one piece.
It is strict; it preserves cap features.
When it decreases, features are reversed horizontally.
When it increases, they are not reversed.
We consider $\mu_n$ to be defined by
Figure~\ref{f10}:
the copy should be squeezed into an interval so narrow that the copy avoids colliding with cones or other parts of the cap;
and the image of the key box should be in place to support the $u_{n+1}$-high box of the outer strand.
Copy $n$ is the image of the old quasicap part under transformation
$(x,y) \mapsto (\mu_n(x), \varphi_n(y)).$
\end{enumerate}
\end{defn}

We claim without proof these facts about 
transformation $(x,y) \mapsto (\mu_n(x), \varphi_n(y))$:
\begin{enumerate}
\item the image of a box is a box, and
\item when two boxes meet in a set of area 0, so do their images.
\end{enumerate}

Figure~\ref{f10} suggests that $u_{N+1} \ge 0.$
Let us prove it.

\begin{lemma}
The hypothesis of Proposition~\ref{p3} implies $u_{N+1} \ge 0$.
\end{lemma}

\begin{proof}
\begin{eqnarray*}
u_{N+1} &=& (r-\theta)u_N - (r-2\theta)u_{N-1} - \theta u_{N-2} \text{ (recurrent definition of $u_{N+1}$)}\\
&=& (r-\theta)(u_N - u_{N-1}) + \theta(u_{N-1} - u_{N-2}) \ge 0.
\end{eqnarray*}
\end{proof}

\begin{lemma}\label{imagebottom}
Transformation $\varphi_n$ maps the bottom of the old key box to depth $ru_{n+1}$ in the new object for $n=1,\ldots, N-1$.
\end{lemma}

\begin{proof}
In the old quasicap, the key box bottom is at depth
$r$.
Therefore under transformation $\varphi_n,$ the new bottom is at depth
\[\varphi_n(r) = (u_{n+1}-u_n)r + ru_n = ru_{n+1}.\]
\end{proof}

The following lemma will be used twice in the proof of the lemma after it.
\begin{lemma}\label{shiftscale}
For $n = 1, \ldots, N-1,$ 
\[\varphi_n(r-\theta) = u_{n+2} +(r-\theta) (u_{n}-u_{n-1}) + ru_{n-1}.\]
\end{lemma}

\begin{proof}
\begin{eqnarray*}
\varphi_n(r-\theta) &=& (u_{n+1}-u_{n})(r-\theta) + ru_{n} \text{ (definition of $\varphi_n$)}\\
&=& (r-\theta)u_{n+1} - (r-2\theta)u_{n} -\theta u_{n-1} -\theta u_{n} + \theta u_{n-1} + ru_{n} \\
&=& u_{n+2} -\theta u_{n} + \theta u_{n-1} + ru_{n} \text{ (recurrent definition of $u_{n+2}$)}\\
&=& u_{n+2} +(r-\theta) (u_{n}-u_{n-1}) + ru_{n-1}.
\end{eqnarray*}
\end{proof}

\begin{lemma}\label{imagetop}
Transformation $\varphi_n$ maps the top of the old key box (depth $r-\theta$ in the old quasicap)
to depth $r + u_3 + u_4 + \cdots +u_{n+2}$ in the new object for $n=1,\ldots,N-1.$
\end{lemma}

\begin{proof}
By induction on $n.$
In case $n=1$,
\begin{eqnarray*}
\varphi_1(r-\theta) 
&=& u_3 +(r-\theta) (u_1-u_0) + ru_0 \text{ (by Lemma~\ref{shiftscale})}\\
&=& u_3 + r \text{ (because $u_0 = u_1 = 1$)}.
\end{eqnarray*}
In case $n>1,$
\begin{eqnarray*}
\varphi_n(r-\theta) 
&=& u_{n+2} +(r-\theta) (u_{n}-u_{n-1}) + ru_{n-1}  \text{ (by Lemma~\ref{shiftscale})}\\
&=& u_{n+2} +\varphi_{n-1}(r-\theta)  \text{ (definition of $\varphi_{n-1}$)}\\
&=& u_{n+2} + r + u_3 + u_4 + \cdots + u_{n+1}  \text{ (by inductive hypothesis)}.
\end{eqnarray*}
\end{proof}

\begin{lemma}\label{suppouter}
The $u_{n+1}$-tall box of the outer strand of the new object is supported for $n=1,\ldots,N-1.$ 
\end{lemma}
\begin{proof}
Support of the $u_{n+1}$-tall box requires coverage of the vertical extent from its bottom (depth $r+u_3 + \cdots+ u_{n+1}$) to its cone top (depth $ru_{n+1}$).
The neighboring $u_{n+2}$-tall box covers the vertical extent from depth $r+u_3+\cdots+u_{n+1}$ to depth $r+u_3+\cdots+u_{n+2}.$
The image under $\varphi_n$ of the old key box covers the rest, i.e., the vertical extent from depth $r+u_3+\cdots+u_{n+2}$ (Lemma~\ref{imagetop}) to
depth $ru_{n+1}$ (Lemma~\ref{imagebottom}).
\end{proof}

\begin{lemma}\label{suppouterlast}
The $u_{N+1}$-tall box of the outer strand of the new object is supported.
\end{lemma}

\begin{proof}
Like the $u_0$- and $u_1$-tall boxes,
and unlike
the $u_n$-tall box for $n=2,\ldots,N,$ 
the $u_{N+1}$-tall box has only one supporting box.
The image under transformation $\varphi_N$ of the old key box has its top at depth $r+u_3+\cdots + u_{N+1}$ by Lemma~\ref{imagetop} (level with the bottom of the $u_{N+1}$-tall box by definition of the new object) and its bottom at depth $ru_N$ (the depth of the cone under the $u_{N+1}$-tall box by definition of the new object).
\end{proof}

\begin{lemma}\label{nooverlap}
Distinct boxes of the new object meet pairwise in sets of area 0.
\end{lemma}

\begin{proof}
For most box pairs, the conclusion is immediate:
\begin{itemize}
\item 2 boxes from the outer strand: by definition of the new object
\item 2 boxes from a common copy: because the old object is a quasicap
\item 2 boxes from distinct copies: by definition of $\mu_1,\ldots,\mu_N$.
\end{itemize}
In the remaining case, one box is from the outer strand, and the other is from copy $n\in\{1,\ldots,N\}$.
No box top of copy $n$ is higher than the top of the key box of the copy, by definition of quasicap.

In case $n=N$, this means no box of copy $n$ extends above the bottom of the $u_{N+1}$-tall box, as a consequence of the proof of Lemma~\ref{suppouterlast}.
So no box of the copy meets the $u_{N+1}$-tall box in a set of positive area.
Boxes of copy $n$ are kept away from the remaining boxes of the outer strand by definition of $\mu_n.$

In cases $n=1,\ldots,N-1,$ no box of copy $n$ extends above the bottom of the $u_{n+2}$-tall box, as a consequence of the proof of Lemma~\ref{suppouter}.
Boxes of copy $n$ do not meet the $u_n$- or $u_{n+1}$-tall boxes, which are higher than the $u_{n+2}$-tall box by definition of the new object.
They are kept away from the remaining boxes of the outer strand by definition of $\mu_n$.
\end{proof}

\begin{lemma}
Except for the twin 1-tall boxes, no box of the new object has its top higher than that of the key box.
\end{lemma}

\begin{proof}
This is a consequence of the proof of Lemma~\ref{nooverlap}.
\end{proof}

\begin{lemma}\label{sparseouter}
The sparseness condition holds for each box of the outer strand.
\end{lemma}

\begin{proof}
The $u_n$-tall box of the outer strand, $n\in\{0,\ldots,N\}$, 
has its cone placed at depth $ru_n$, and there are no other boxes between cap top and cone top. 
The $u_{N+1}$-tall box has its cone at depth $ru_N \ge ru_{N+1}.$
\end{proof}

\begin{lemma}\label{sparse}
The sparseness condition holds for each box in copy $n\in\{1,\ldots,N-1\}$.
\end{lemma}

\begin{proof}
The goal is to prove the sparseness condition (as in the definition of rational cap) for a box $B'$ in copy $n$.
Box $B'$ is the image under $\varphi_n$ of a box $B$ in the old quasicap. In the old quasicap, let
$d$ be the distance from cap top to cone top of $B$, and let $w$ be the weight above its cone.
In the new structure, let $d'$ be the distance from cap top to cone top of $B'$, and let $w'$ be the weight above its cone.

The old structure is a quasicap. Sparseness holds for $B$, so $d \ge rw$.
Now 
\begin{eqnarray*}
d' &=& \varphi_n(d) \\
&=& (u_{n+1}-u_{n})d + ru_{n}\\
&\ge& r[(u_{n+1}-u_{n})w + u_{n}]
\end{eqnarray*}
because $u_{n+1}\ge u_{n}.$

Now we compute weight above the cone of $B'$.
First observe the effect of $\varphi_n$ on a single box.
Suppose a box in the old quasicap has top at depth $y$ and bottom at depth $y+\Delta y$.
In the new structure, its image under $\varphi_n$ has top at depth
\[\varphi_n(y) = (u_{n+1} - u_n)y - ru_n\]
and bottom at depth
\[\varphi_n(y + \Delta y) = (u_{n+1}-u_n)(y+ \Delta y) - ru_n.\]
The difference is $(u_{n+1}-u_n)\Delta y.$
That is, the image under $\varphi_n$ is $u_{n+1}-u_n$ times taller than the original.

The weight $w'$ above the cone of $B'$ in the new structure is similarly related to $w$,
for all the boxes of the old quasicap except the twin ones appear in the new structure, stretched by the same factor.
There is also in the new structure a new box of the outer strand contributing weight above.

There are two kinds of box $B$ in the old quasicap to consider: those under the 1-tall box (the \emph{1 side} in Figure~\ref{f09}), and those not (the \emph{key side} in Figure~\ref{f09}).
In case $B$ is on the 1 side, all the weight above its cone except for 1 unit (the twin) is mapped into the new structure,
scaled by a factor of $u_{n+1}-u_n,$ 
and then a $u_{n+1}$-tall box is added.
That is, the weight above the cone of $B'$ in the new structure is
\begin{eqnarray*}
w' &=& (u_{n+1}-u_{n})(w-1) + u_{n+1}\\
&=& (u_{n+1}-u_{n})w + u_{n}.
\end{eqnarray*}
In case $B$ is on the key side, all the weight above its cone is mapped into the new structure, scaled by a factor of $u_{n+1}-u_n,$
and then a $u_{n}$-tall box is added.
That is, the weight above the cone of $B'$ in the new structure is
\begin{eqnarray*}
w' &=& (u_{n+1}-u_{n})w + u_{n},
\end{eqnarray*}
same as when $B$ is on the 1 side.
In both cases, $d' \ge rw'.$
Sparseness holds.
\end{proof}

\begin{lemma}\label{sparseN}
The sparseness condition holds for each box in copy $N$.
\end{lemma}

\begin{proof}
Refer to Figure~\ref{f10} and compare copy $N$ to copy $N-1,$ as $\varphi_N = \varphi_{N-1}.$
Distances from cap top to cone top are equal in corresponding cones of the two copies.
Now consider weight above.
It is the same in the two copies except for contributions of outer-strand boxes.
Each cone of copy $N$ either has a $u_{N+1}$-tall outer-strand box above, while the corresponding cone of copy $N-1$ has a $u_N$-tall box;
or it has no outer-strand box above, while the corresponding cone has a $u_{N-1}$-tall box.
Because $u_{N+1} \le u_N$, in both cases there is no more weight above cones of copy $N$ than above corresponding cones of copy $N-1.$
\end{proof}

The proof of Proposition~\ref{p3} is complete.
New lower bounds on $R$ are available.
For example, let $r=4.5.$
To start a sequence of quasicaps, we construct a $1$-quasicap.
This is easy. We put a 1-tall box in the key position and a cone directly beneath it.
If $(\theta, \delta)= (1, 0.8)$, then sequence $u_0, u_1, u_2, \ldots$ begins 
\[1 = 1 \le \frac{9}{5} < \frac{14}{5} < \frac{43}{10} < \frac{25}{4} < \frac{333}{40} < \frac{737}{80} \ge \frac{829}{160},\]
so by Proposition~\ref{p3} there exists a $1.8$-quasicap.
If $(\theta, \delta) = (1.8, 0.7)$, then sequence $u_0, u_1, u_2, \ldots$ begins
\[1 = 1 \le \frac{5}{2} < \frac{81}{20} < \frac{1377}{200} < \frac{20889}{2000} < \frac{294273}{20000} < \frac{3586761}{200000} \ge \frac{32757777}{2000000},\]
so by Proposition~\ref{p3} there exists a $2.5$-quasicap.
This is a $4.5$-cap, so by Proposition~\ref{tcap} we have $4.5 \le R.$
In the next section we show that the construction works for all $r<5.$

\section{The special sequence}\label{rseq}

Sequence $u_n$ of Section~\ref{s3} is a linear homogeneous recursive sequence. 
So is its sequence of differences.
The ordinary power series generating function 
\[f(x)=\sum_{n\ge 0}(u_{n+1}-u_n)x^n\]
of the difference sequence is
(cf.\ chapter 4 of \cite{Stanley})
\[f(x) =  \frac{p(x)}{q(x)},\] 
where 
$q(x)$ is a polynomial whose coefficients are those of the recurrence,
and $p(x)$ is a polynomial of lesser degree related to boundary values of the sequence.
Specifically,
\[p(x) = x[(\theta+\delta)(1-x)-1]\]
and
\begin{eqnarray}
q(x) &=& 1-(r-\theta)x+(r-2\theta)x^2+\theta x^3 \label{Q1}\\
&=& 1+rx(x-1)+\theta x(x-1)^2 \label{Q2} \\
&=& (1-x/\alpha)(1-x/\beta)(1-x/\gamma) \label{Q3}.
\end{eqnarray}
Form~\eqref{Q3} gives names $\alpha, \beta, $ and $\gamma$
to the (complex) roots of $q(x).$
They affect the asymptotic behavior of the sequence.
When they are distinct,
\[f(x) = \frac{A}{1-x/\alpha} + \frac{B}{1-x/\beta} + \frac{C}{1-x/\gamma}\]
and \begin{equation}\label{foob} 
u_{n+1}-u_n = A\alpha^{-n}+B\beta^{-n}+C\gamma^{-n}
\end{equation}
for some complex numbers $A,B,C.$ Using \eqref{Q3},
\[q(x)f(x) = A(1-x/\beta)(1-x/\gamma) + B(1-x/\alpha)(1-x/\gamma) + C(1-x/\alpha)(1-x/\beta).\]

\begin{prop}\label{roots}
When $1 \le \theta \le 0.5r \le 3\theta,$ we can assume $\alpha<-1$
and $0<\beta\gamma<1.$
\end{prop}
\begin{proof}
The first conclusion follows from the intermediate value theorem and
\eqref{Q2}:
\begin{eqnarray*}
q(-1) &=& 1+2r -4\theta >0\\
q(-10) &=& 1 + 110 r -1210 \theta <0.
\end{eqnarray*}
For the second conclusion, compare
cubic terms of \eqref{Q1} and \eqref{Q3}:
\[-\alpha\beta\gamma =\theta^{-1}\le 1.\]
\end{proof}
For the rest of the section we restrict $r$ and $\theta$:
\[4.999\le r\le 5\]
\[1 \le \theta \le 2.2.\]
So Proposition~\ref{roots} applies.
This simplifies our calculations.
We may ignore $\alpha$ in the asymptotic analysis because the corresponding term of \eqref{foob} vanishes.

\begin{prop}\label{realg}
If $\theta\le 2.13,$ then $\beta$ and $\gamma$ are real, and
we can assume $0<\gamma<0.56<\beta<1.$
\end{prop}

\begin{proof}
By the intermediate value theorem.
Using \eqref{Q2},
\[q(0) = q(1) =1>0\] and 
\[q(0.56) \le 1-4.999\cdot 0.56 \cdot 0.44 + 2.13\cdot 0.56\cdot 0.44^2<0.\]
\end{proof}

\begin{prop}\label{real}
If $\beta$ and $\gamma$ are real, and
$\alpha<-1<0<\gamma<\beta<1,$
then $u_{n+1}-u_n\to -\infty$ when \[\theta+\delta<\frac{1}{1-\gamma}.\]
\end{prop}
\begin{proof}
In this case \[u_{n+1}-u_n \sim  C\gamma^{-n},\]
and the desired conclusion follows when $C<0$.
Comparing two expressions of $q(x)f(x)$,
\begin{eqnarray}
x[(\theta+\delta)(1-x)-1] &=& A(1-x/\beta)(1-x/\gamma) + \label{coeffs}\\ 
&& B(1-x/\alpha)(1-x/\gamma) + \notag \\
&& C(1-x/\alpha)(1-x/\beta). \notag
\end{eqnarray}
Substituting $\gamma$ for $x,$
\[\gamma[(\theta+\delta)(1-\gamma)-1] = C(1-\gamma/\alpha)(1-\gamma/\beta).\]
Because \[0<(1-\gamma/\alpha)(1-\gamma/\beta),\]
we have $C<0$ when \[(\theta+\delta)(1-\gamma)-1 <0.\]
\end{proof}

\begin{prop}\label{decreasing}
If $\theta\le 2.13,$ then
$\gamma$ decreases strictly in $r$, and so does $\displaystyle \frac{1}{1-\gamma}-\theta.$
\end{prop}
\begin{proof}
Suppose \[4.999\le r_0<r_1\le 5.\] 
Using \eqref{Q2}
for $j\in\{0,1\},$ 
\[q_j(x) := 1+r_j x(x-1)+\theta x(x-1)^2\]
has roots $0<\gamma_j<\beta_j<1$ by Proposition~\ref{realg}.
When $\gamma_0 \le x \le \beta_0,$ we have  
\[0\ge q_0(x)=1+r_0 x(x-1) +\theta x(x-1)^2\]
and
\[0 > (r_1-r_0) x(x-1).\] Adding the last two inequalities,
\[0 > 1+r_1 x(x-1) +\theta x(x-1)^2=q_1(x).\]
Specifically, $q_1(\gamma_0)<0$.
Because $q_1(x)\ge 0$ for $x\in [0,\gamma_1]\cup [\beta_1,1],$ it follows 
that $\gamma_1<\gamma_0.$ That is, $\gamma$ decreases in $r.$
The second conclusion follows.
\end{proof}

\begin{prop}\label{pos}
If $\theta\le 2.13$ and $r=5,$ then $\displaystyle \theta\le \frac{1}{1-\gamma}.$
\end{prop}

\begin{proof}
Using \eqref{Q2},
\[q(1-\theta^{-1}) = 1+5(1-\theta^{-1})(-\theta^{-1})+\theta(1-\theta^{-1})\theta^{-2}=  \left(1-2\theta^{-1}\right)^2\ge 0.\]
By Proposition~\ref{realg} and the surrounding discussion, $1-\theta^{-1}\le \gamma$ or 
$1-\theta^{-1} \ge \beta > 0.56.$ The latter contradicts the hypothesis $\theta \le 2.13$.
So $1-\theta^{-1}\le \gamma.$
\end{proof}

\begin{prop}\label{last}
If $\theta=2.13,$ then $\displaystyle \frac{1}{1-\gamma}-\theta > 0.04.$
\end{prop}

\begin{proof}
Proposition~\ref{decreasing} implies that $\gamma$ is least when $r=5$. 
Evaluate $q(0.54)$ there (using \eqref{Q2}) to obtain a lower bound for 
$\gamma$:
\[q(0.54) = 1-5\cdot 0.54\cdot 0.46 + 2.13\cdot 0.54\cdot 0.46^2 >0.\]
So $\gamma > 0.54$ and $\displaystyle \frac{1}{1-\gamma}-\theta > 0.04.$
\end{proof}

The discriminant $D$ of $q(x)$ is (cf.\ pp.\ 95-102 of \cite{Rotman})
\[D = -27\theta^2 -4\theta^3 +6\theta^2r +6 \theta r^2
+\theta^2 r^2 -4r^3 -2\theta r^3 +r^4,\]
and $D<0$ if and only if $\text{Im}[\gamma]\text{Im}[\beta]\ne 0.$
\begin{prop}\label{imaginary}
If $\theta=2.15,$ then $D<0.$
\end{prop}
\begin{proof}

If $\theta \ge 2.1,$ then $D$ increases in $r$:
\begin{eqnarray*}
\frac{dD}{d r} &=& 6\theta^2+12\theta r +2\theta^2 r -12r^2 -6\theta r^2 +4r^3\\
&\ge& 6\cdot 2.1^2 +12\cdot 2.1\cdot 4.9 + 2\cdot 2.1^2\cdot 4.9-12\cdot 5^2 - 6\cdot 2.2\cdot 5^2+ 4\cdot 4.9^3 > 0.
\end{eqnarray*}
So $D$ is greatest when $r=5.$
\begin{eqnarray*}
D &=& -27\cdot 2.15^2 -4\cdot 2.15^3 +6\cdot 2.15^2\cdot 5 + 6\cdot 2.15\cdot 5^2 +\\
&&2.15^2\cdot 5^2 - 4\cdot 5^3 - 2\cdot 2.15\cdot 5^3 + 5^4 <0.
\end{eqnarray*}

\end{proof}

\begin{prop}\label{complex}
If $D<0$ and $|\alpha|>1>|\beta|=|\gamma|>0,$ then $u_N\ge u_{N+1}$ 
for some $N>0.$
\end{prop}

\begin{proof}
In this case, \eqref{foob} is equivalent to:
\[u_{n+1}-u_n = A\alpha^{-n}+ 2\text{Re}[C\gamma^{-n}],\]
because
formula 
\eqref{coeffs} for $A,B,$ and $C$ is symmetric, 
and complex conjugation is a field automorphism.

Clearly $A\alpha^{-n}\to 0$ as $n\to\infty.$ The behavior of the other
term $2\text{Re}[C\gamma^{-n}]$ can be understood with the help of the
Euler formula \[z=|z|\exp{(i\zeta)},\] where the \emph{modulus} $|z|=\sqrt{x^2+y^2}$ of complex number $z=x+iy$ is the distance in the complex plane of $z$ from $0$, and the \emph{argument} $\zeta$ of $z$ is an angle in the complex plane 
from the positive real axis to the ray emanating from $0$ to $z$.
Complex multiplication is multiplicative in modulus and 
additive in argument.
Because $|C\gamma^{-n}|\to \infty$ as $n\to \infty,$ the desired conclusion holds if the ray emanating from $0$ to $C\gamma^{-n}$ is near the negative real
axis for some $n$ large enough that the second term $2\text{Re}[C\gamma^{-n}]$
 of the sum is dominant.

Indeed, $2\text{Re}[C\gamma^{-n}]$ oscillates in sign.
Because $|\alpha|>1,$ eventually $|A\alpha^{-n}|<0.1,$ say, while
$2\text{Re}[C\gamma^{-n}]<-0.1$ infinitely often (cf.\ exercise II.1.11 of \cite{Conway}).

\end{proof}

\begin{thm}\label{fivelter}
$5\le R.$
\end{thm}

\begin{proof}
Fix some rational $r$ with $4.999<r<5.$
Let $\Theta$ be the set of $\theta$ such that a $\theta$-quasicap exists.
Clearly $1\in \Theta.$
Recall that $\gamma$ depends on $r$ and $\theta.$ Though $r$ is fixed,
$\theta$ is free, so $\gamma$ is a function of $\theta.$ Let
\[F = \left\{\frac{1}{1-\gamma} - \theta \middle\vert 1 \le \theta \le 2.13\right\}.\]
Then $F$ is closed (in the usual topology on the reals) because it is the continuous image of a compact set.
Propositions~\ref{decreasing} and \ref{pos} imply $\inf F > 0.$
Fix some rational $\delta$ with $0<\delta<\inf F$ so that $(2.13-1)/\delta$
is an integer.
Propositions~\ref{roots}, \ref{realg}, and \ref{real} imply 
$u_{n+1}-u_n \to -\infty$ for all $\theta\in[1,2.13].$ 
Invoking Proposition~\ref{p3} many (i.e., $1.13/\delta$) times yields 
$2.13\in \Theta.$
Propositions~\ref{p3} and \ref{last} imply $2.15\in \Theta$ 
because $2.15-2.13 < 0.04.$ 
Propositions~\ref{p3}, \ref{imaginary}, and \ref{complex} imply $r-2\in \Theta.$
That is, a rational $r$-cap exists. Proposition~\ref{tcap} implies $5\le R.$ 
\end{proof}

For some bad pairs $(r,\theta),$ no positive $\delta$ is small
enough for Proposition~\ref{p3}. 
In the present section, 
we worked to avoid such sequences while getting as close as 
possible.
The curve $\frac{1}{1-\gamma}=\theta$ with $\delta=0$ in the $(r,\theta)$-plane
(see Figure~\ref{extreme})
is an important part of the boundary of the bad region. 
What is the least $r$ on this curve?
We aim to express $r$ as a function of $\theta$ and optimize.
The curve is defined
$\theta=\frac{1}{1-\gamma},$ so
$1 = \theta(1-\gamma).$
Now
\begin{eqnarray*}
0&=& q(\gamma)\\
&=& 1-(r-\theta)\gamma+(r-2\theta)\gamma^2+\theta \gamma^3\\
&=& 1-r(\gamma-\gamma^2)+\theta(\gamma-2\gamma^2+\gamma^3)\\
r&=& \frac{1+\theta\gamma(1-\gamma)^2}{\gamma(1-\gamma)}\\
&=& 1+\theta/\gamma\\
&=& 1+\theta(1-\theta^{-1})^{-1}\\
\frac{d r}{d \theta} &=& (1-\theta^{-1})^{-1}-\theta^{-1}(1-\theta^{-1})^{-2},
\end{eqnarray*}
and the minimum value of $r$ on the curve occurs when 
$dr/d\theta = 0, \theta = 2,$ and $r=5.$
The $r$-minimal bad sequence $u_n$ is the one with
$(r,\theta,\delta)=(5,2,0)$.
\begin{dasnote} If $(r,\theta,\delta)=(5,2,0),$ then 
$u_n$ is the Fibonacci sequence.
\end{dasnote}

\begin{figure}
\setlength{\unitlength}{1mm}
\begin{picture}(80,98)(-35,-48)
\linethickness{1pt}
\qbezier(10,20)(-10,0)(10,-20)
\put(-40,-40){\line(1,0){93}}
\put(-40,-40){\line(0,1){83}}
\multiput(-40,-40)(0,40){3}{\line(-1,0){1}}
\multiput(-40,-40)(40,0){3}{\line(0,-1){1}}

\put(-46,40){\makebox(0,0){$2.01$}}
\put(-46,0){\makebox(0,0){$2.00$}}
\put(-46,-40){\makebox(0,0){$1.99$}}
\put(-42,47){\makebox(0,0){$\theta$}}

\put(-40,-44){\makebox(0,0){$4.9999$}}
\put(  0,-44){\makebox(0,0){$5.0000$}}
\put( 40,-44){\makebox(0,0){$5.0001$}}
\put( 56,-44){\makebox(0,0){$r$}}

\put(-7,15){\makebox(0,0){$\frac{1}{1-\gamma}>\theta$}}
\put(18,15){\makebox(0,0){$\frac{1}{1-\gamma}<\theta$}}

\put(42,0){\makebox(0,0){Fibonacci sequence}}
\put(23,0){\vector(-1,0){20}}
\put(0,0){\circle*{1.6}}
\put(-40,0){\line(1,0){40}}
\put(0,-40){\line(0,1){40}}
\put(14,-11){\makebox(0,0){bad}}
\put(22,-15){\makebox(0,0){sequences}}
\end{picture}
\caption{The Fibonacci sequence $(r,\theta)=(5,2)$ is $r$-minimally bad.}
\label{extreme}
\end{figure}
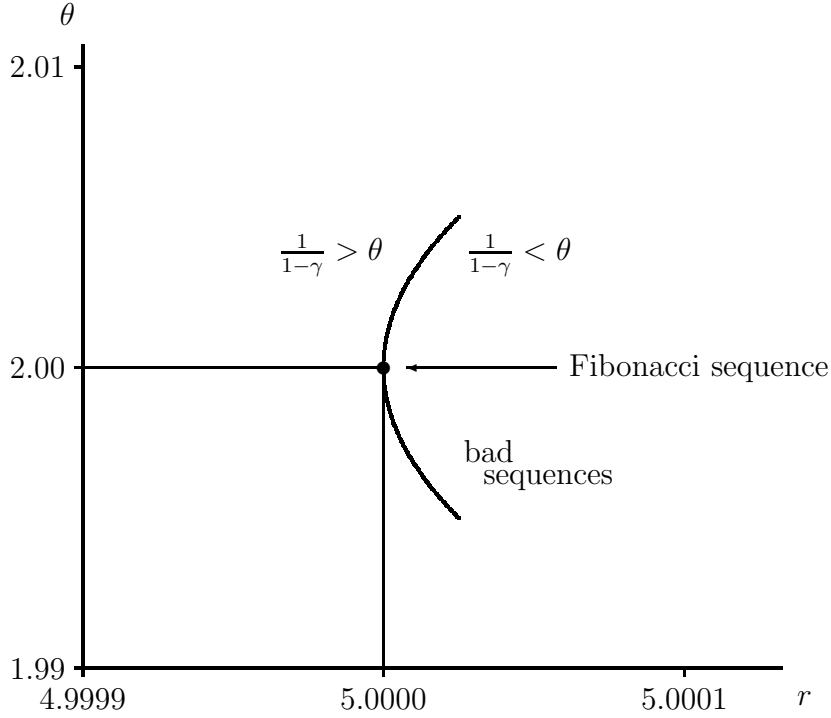

\section{No binary 5-cap exists}\label{nocap}

Caps of Section~\ref{s3} are like ordered binary trees (see Figure~\ref{f12}):
\begin{enumerate}
\item A cap has at least one box.
\item Each box but one (the \emph{top} box) supports a unique box.
\item Each box has at most 1 left supporter and at most 1 right supporter.
\item Each box is a member of the set $B$ of boxes that are either the top box or a supporter of a member of $B$.
\end{enumerate}
In case a box has both a left and right supporter, the bottom of one supporter (the \emph{high} supporter) is level 
with the top of the other (the \emph{low} supporter).
Further, in our construction:
\begin{enumerate}
\item Each high supporter of a left supporter is a left supporter.
\item Each high supporter of a right supporter is a right supporter.
\end{enumerate}
A cap of finitely many boxes that satisfies these conditions is a \emph{binary cap}.
We will show that the copy-and-paste method of Section~\ref{s3}
yields the best possible binary caps.

\psfrag{lambda}{$\lambda$}
\psfrag{H}{$H$}
\psfrag{HH}{$H^2$}
\psfrag{HHH}{$H^3$}
\psfrag{L}{$L$}
\psfrag{LH}{$LH$}
\psfrag{LHH}{$LH^2$}
\psfrag{HL}{$HL$}
\psfrag{HLH}{$HLH$}
\psfrag{LL}{$L^2$}
\begin{figure}
\begin{center}\includegraphics{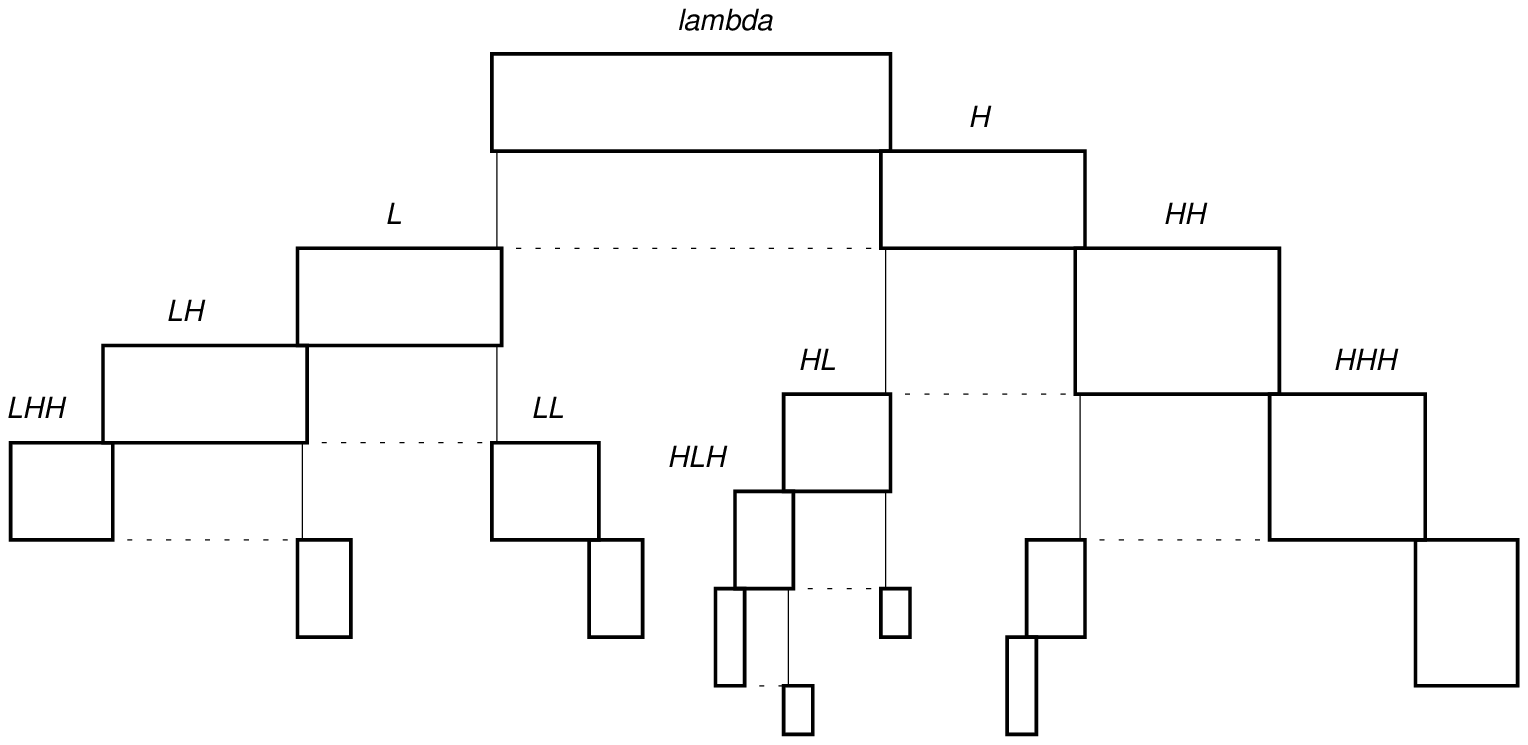}\end{center}
\caption{A binary cap }\label{f12}
\end{figure}

We name each box
of binary cap $C$ by a binary word, as in Figure~\ref{f12}.
Box $w$ may have a 
high supporter, its \emph{high child} $wH$;
and a low supporter, its \emph{low child} $wL$.
Accordingly $w$ is the \emph{parent} of $wH$ and $wL$.
Indeed, we may describe the whole binary cap $C$ by a function 
$\kappa\colon \{H,L\}^*\to \mathbb{R}^+$ where $C$ has a box
$w\in\{H,L\}^*$ if and only if $\kappa(w)>0$. 
The word of length $0$ is denoted $\lambda.$ 
We assume $\kappa(w)>0$ whenever $\kappa(wH)>0$ or $\kappa(wL)>0$.
Binary caps are nonempty, and in particular, $\kappa(\lambda)>0.$

Let $\kappa_w=\kappa(w).$
Let $\tau_w$ denote the depth of the top of the box
$w$.
Let $\beta_w$ denote the weight above the cone of box $w$ minus
the height of box $w$ itself.
Let $\alpha_w$ denote the weight above the cone of the parent of $w$.
Let $\pi_w$ (the \emph{penalty} of $w$) be $\alpha_w - \beta_w.$
Let $(f_0,f_1,\ldots)=(0,1,\ldots)$ where
$f_n=f_{n-1}+f_{n-2}$ (the Fibonacci recurrence) for $n\ge 2.$

The following hold for boxes $\lambda,v,vL,vH$ in a binary $r$-cap.

\begin{align}
0 & =\beta_{\lambda}=\pi_{\lambda}=\tau_{\lambda};\tag{Cap-\ensuremath{\lambda}}\\
\notag\\
\beta_{vL} & =\beta_{v}+\pi_{v};\tag{L-\ensuremath{\beta}}\\
\pi_{vL} & =\kappa_{v}-\pi_{v};\tag{L-\ensuremath{\pi}}\\
\kappa_{vL} & \geq r(\beta_{v}+\kappa_{v})-\tau_{vL}\geq0;\tag{L-\ensuremath{\kappa}}\\
\notag\\
\beta_{vH} & =\beta_{v};\tag{H-\ensuremath{\beta}}\\
\pi_{vH} & =\kappa_{v};\tag{H-\ensuremath{\pi}}\\
\kappa_{vH} & \geq\tau_{vL}-\tau_{v}-\kappa_{v}\geq0;\tag{H-\ensuremath{\kappa}}\\
\tau_{vH} & =\tau_{v}+\kappa_{v}.\tag{H-\ensuremath{\tau}}
\end{align}

\begin{defn}
Suppose $w=uH^{m}$ is a box in a binary cap, and set $w_{i}=uH^{i}$ for
$0 \le i \le m$.
Then $w$ is $m$-\emph{hard} if for all $i$ with $0\leq i\leq m$
each of the following holds true:
\begin{enumerate}
\item [(H0)]$\pi_{w_{i}}\geq0$;
\item [(H1)]$\kappa_{w_{i}}\geq2\pi_{w_{i}}-\pi_{u}f_{i}$; and 
\item [(H2)]$\tau_{w_{i}}\leq5\beta_{w_{i}}+2\pi_{w_{i}}+\pi_{u}f_{i+1}$. 
\end{enumerate}
\noindent A box is \emph{hard} if it is $m$-hard for some $m\in\{0,1,\dots\}$.
\end{defn}
\begin{lemma}
\label{lem:hard}If $w=uH^{m}$ is $m$-hard then $\kappa_{w}\geq\pi_{u}f_{m+3}$.
\end{lemma}
\begin{proof}
Argue by induction on $m$. If $m=0$ then $\kappa_{w}\geq2\pi_{w}-\pi_{u}f_{0}=\pi_{u}f_{3}$.
Else $m>0$. Then $w'=uH^{m-1}$ is $(m-1)$-hard. By (H1), (H-$\pi$)
and induction, we have:
\[
\kappa_{w}\geq2\pi_{w}-\pi_{u}f_{m}=2\kappa_{w'}-\pi_{u}f_{m}\geq2\pi_{u}f_{m+2}-\pi_{u}f_{m}=\pi_{u}(f_{m+2}+f_{m+1})=\pi_{u}f_{m+3}.
\]
\end{proof}
\begin{thm}
\noindent There does not exist a binary $5$-cap.\end{thm}
\begin{proof}
\noindent Suppose $\kappa$ is a binary $5$-cap.  The top box $\lambda$
is $0$-hard as $\kappa_{\lambda}>0$ and $\pi_{\lambda}=\tau_{\lambda}=0$.
Since caps are finite, there is a hard box $w$ with maximum word
length. For a contradiction, we show that there exists $w^{+}\in\{wL,wH\}$
such that $w^{+}$ is hard and $\kappa_{w^{+}}>0$ (so $w^{+}$ is
in $\kappa$). 

Let $w=uH^{m}$ be $m$-hard. If $\tau_{wL}<5\beta_{w}+3\kappa_{w}+2\pi_{w}$
then $wL$ is $0$-hard, since

\begin{align*}
\pi_{wL} & =\kappa_{w}-\pi_{w}\geq2\pi_{w}-\pi_{w}\geq0 &  & \mbox{((L-\ensuremath{\pi}), (H1), (H0))}\\
\intertext{and}\kappa_{wL} & \ge5\beta_{w}+5\kappa_{w}-\tau_{wL} &  & \text{((L-\ensuremath{\kappa}))}\\
 & >5\beta_{w}+5\kappa_{w}-(5\beta_{w}+3\kappa_{w}+2\pi_{w}) &  & \text{(case)}\\
 & =2\kappa_{w}-2\pi_{w}\\
 & =2\pi_{wL}-\pi_{wL}f_{0} &  & \text{((L-\ensuremath{\pi}), Fibonacci)}\\
\intertext{and}\tau_{wL} & <5\beta_{w}+3\kappa_{w}+2\pi_{w}+3(\pi_{w}-\pi_{w}) &  & \text{(case)}\\
 & =5\beta_{wL}+2\pi_{wL}+\pi_{wL}f_{1}. &  & \text{((L-\ensuremath{\beta}), (L-\ensuremath{\pi}), Fibonacci)}
\end{align*}
Moreover, $\kappa_{wL}>2\pi_{wL}-\pi_{wL}f_{0}=2\pi_{wL}\geq0$, since
in this case (H1) is strict. 

Otherwise $\tau_{wL}\geq5\beta_{w}+3\kappa_{w}+2\pi_{w}$. Then $wH=uH^{m+1}$
is $(m+1)$-hard, since 
\begin{align*}
\pi_{w} & =\kappa_{uH^{m}}\geq0 &  & \mbox{((H-\ensuremath{\pi}))}\\
\intertext{and}\kappa_{wH} & \ge\tau_{wL}-\tau_{w}-\kappa_{w} &  & \text{((H-\ensuremath{\kappa}))}\\
 & \geq(5\beta_{w}+3\kappa_{w}+2\pi_{w})-(5\beta_{w}+2\pi_{w}+\pi_{u}f_{m+1})-\kappa_{w} &  & \text{(case, (H2) for \ensuremath{w})}\\
 & =2\kappa_{w}-\pi_{u}f_{m+1}\\
 & =2\pi_{wH}-\pi_{u}f_{m+1} &  & \text{((H-\ensuremath{\pi}))}\\
\intertext{and}\tau_{wH} & =\tau_{w}+\kappa_{w} &  & \text{((H-\ensuremath{\tau}))}\\
 & \leq5\beta_{w}+2\pi_{w}+\pi_{u}f_{m+1}+\kappa_{w} &  & \text{((H2) for \ensuremath{w})}\\
 & \le5\beta_{wH}+\kappa_{w}+(2\pi_{w}-\pi_{u}f_{m})+\pi_{u}f_{m+2} &  & \text{((H-\ensuremath{\beta}), Fibonacci)}\\
 & \leq5\beta_{wH}+2\kappa_{w}+\pi_{u}f_{m+2} &  & \mbox{((H1) for \ensuremath{w})}\\
 & \leq5\beta_{wH}+2\pi_{wH}+\pi_{u}f_{m+2}. &  & \text{((H-\ensuremath{\pi}))}
\end{align*}
If $\pi_{u}>0$ then $\kappa_{wH}\geq\pi_{u}f_{m+4}>0$ by Lemma~\ref{lem:hard}.
If $\pi_{u}=0$ then $\kappa_{wH}\geq2\pi_{wH}=2\kappa_{w}>0$.\end{proof}

\end{document}